\documentclass[12pt]{amsart}

\usepackage[english]{babel}

\usepackage[a4paper]{geometry}
\usepackage{mathalfa}
\usepackage{tikz} 
\usepackage{caption}
\usepackage{subcaption}
\usepackage{amsmath,amssymb,amsfonts,amsthm,a4,fleqn,bbm,dsfont,enumerate}
\usepackage{enumitem,cite}
\newtheorem{definition}{Definition}[section] 
\newtheorem{lemma}[definition]{Lemma}
\newtheorem{example}[definition]{Example}
\newtheorem{theorem}[definition]{Theorem}
\newtheorem{proposition}[definition]{Proposition}
\newtheorem{corollary}[definition]{Corollary}
\newtheorem{remark}[definition]{Remark}

\usepackage{graphicx}
\graphicspath{C:\Users\akhil\OneDrive\Desktop\Library Mathematics\Other}
\usepackage[colorlinks=true, allcolors=blue]{hyperref}
\makeatletter
\newcommand{\ostar}{\mathbin{\mathpalette\make@circled*}}
\newcommand{\make@circled}[2]{
	\ooalign{$\m@th#1\smallbigcirc{#1}$\cr\hidewidth$\m@th#1#2$\hidewidth\cr}
}
\newcommand{\smallbigcirc}[1]{
	\vcenter{\hbox{\scalebox{0.77778}{$\m@th#1\bigcirc$}}}
}
\@namedef{subjclassname@2020}{
	\textup{2020} Mathematics Subject Classification}
\makeatother

\begin{document}
	\title{Furtherness in finite topological spaces}
	\author[A. Badra]{Akhilesh Badra}
	\address{AKHILESH BADRA; Department of Mathematics, University of Delhi,
		Delhi--110 007, India}
	\email{akhileshbadra028@gmail.com}
	\author[H. K. Singh]{Hemant kumar Singh}
	\address{HEMANT KUMAR SINGH; Department of Mathematics, University of Delhi,
		Delhi--110 007, India}
	\email{hemantksingh@maths.du.ac.in}
	\thanks{The first author is supported by research grant from the Council of Scientific and Industrial Research (CSIR), Ministry of Science and Technology, Government of India with reference number: 09/0045(13774)/2022-EMR-I}
	

	\maketitle
	\begin{abstract}
		In this paper, we introduce a novel distance-like notion of furtherness for finite topological spaces, demonstrating that every finite space can be viewed as an asymmetric pseudometric space. In particular, we show that every finite $T_0$ space is asymmetric metric space. The topology induced by the forward balls coincides with the original topology of the space, while the backward balls induce the opposite topology. To capture essential information about each finite space, we construct a furtherness Matrix, which gives significant structural details of the finite space. As an application, we introduce the notion of center and radius of subsets of finite topological spaces.

	\end{abstract}

	\section{\textbf{Introduction}}
	
	In metric spaces, the concept of distance provides a fundamental framework for analyzing the relationships between points. However, this notion is not universally applicable in arbitrary topological spaces. A subclass of topological spaces, known as metrizable spaces, allows the definition of distance between points through the existence of a compatible metric. Among finite spaces, only discrete spaces are metrizable, while non-discrete finite spaces inherently lack this property.  
	
	Given the practical utility and simplicity of finite spaces, which find applications across various domains, we propose a distance-like concept, termed "furtherness," defined for all finite spaces. Although this notion of furtherness is  not symmetric, but satisfies the properties of an asymmetric semidistance.
	 We demonstrate that the topologies generated by forward and backward balls align with the original topology and its opposite, respectively. This provides a new lens to  interpret topological structures of finite spaces.\par 
	
	By the end of the $19^{th}$ century (Ref. \cite{Dedekind}), Richard Dedekind provided a characterization of modular lattices of finite length using what is now called a dimension function. The associated distance function reveals that every modular lattice of finite length can be viewed as an asymmetric metric space. In the begining of $20^{th}$ century (Ref. \cite{Birkhoff,Birkhoff2}), Garrett Birkhoff showed that the category of finite posets and the category of finite distributive lattices are essentially equivalent. Around the same time (Ref. \cite{Alexandroff}), Pavel Alexandroff established that finite $T_0$	spaces and finite posets are the same objects viewed from two different perspectives. Since every finite distributive lattice is a modular lattice of finite length, it follows that every finite $T_0$	space admits an associated asymmetric metric defined on it. Our notion of furtherness on finite spaces also verifies this result, particularly in the case of finite $T_0$	spaces. We show that every finite space can be regarded as an asymmetric pseudometric space. Using furtherness, one can prove that every finite $T_0$ space is an asymmetric metric space solely from the open sets of its topology, without any reliance to lattices or posets.\par
	
	Building upon the partial order structure inherent in finite spaces,  we establish a characterization of the preorder $x\leq y$ if and only if $x\in U_y$ (Ref. \cite{Stong}), 
	 within the framework of furtherness. This connection establishes a bridge between furtherness and fundamental concepts in algebraic topology,  offering new insights into the interplay between topological properties and distance-like notions in finite spaces.
	Central to this work is the construction of a furtherness matrix, which captures topological properties for each finite space. \par

		We introduce the concepts of center and radius for subsets of finite topological spaces, which is parallel to the notion in metric spaces (Ref. \cite{Badra}). 
    	This work not only deepens the understanding of finite spaces through furtherness and partial order relations but also opens new pathways for exploring algebraic and topological properties in discrete settings.\par

	Throughout the paper, a finite space $X$ means a finite topological space.

	\section{\textbf{Furtherness function}}	\label{s2.1}

	In this section, we introduce a distance like notion of furtherness between any two points of a finite topological space. Let $X$ be a finite topological space with topology $\mathcal{T}$.
	So, the number of open sets in $\mathcal{T}$ will be finite. Therefore, for every $x\in X$, we can always get atleast one finite sequence of open sets such that $x\in U_0\subseteq U_1\subseteq U_2 \subseteq...\subseteq U_j\subseteq  ... \subseteq  X$.\par
	
	Before moving on we first observe this in a few examples.
	
	\begin{example}\label{2.1}
		{\normalfont
			Let $X=\{1,2,3\}$ be a set and $\mathcal{T}=\{\emptyset,X,\{1,2\}\}$ be the topology on $X$.
			Here, $1\in \{1,2\}\subseteq\{1,2,3\}=X$,  $2\in \{1,2\}\subseteq\{1,2,3\}=X$ and $3 \in\{1,2,3\}=X$.
		}
	\end{example}
	
	\begin{example}\label{2.2}
		{\normalfont
			Let $X=\{a,b,c,d\}$ be a set and $\mathcal{T}=\{\emptyset,X,\{a\},\{d\},\{a,b\},\{a,d\},\\ 
			\{a,b,d\}\}$ be the topology on $X$. 
			Here, $a\in\{a\}\subseteq \{a,b\}\subseteq\{a,b,d\}\subseteq\{a,b,c,d\}=X$, $b\in \{a,b\}\subseteq\{a,b,d\}\subseteq\{a,b,c,d\}=X$, $c\in\{a,b,c,d\}=X$ and $d\in\{d\}\subseteq \{a,d\}\subseteq\{a,b,c,d\}=X$.\par 
			In above examples, we can notice a finite sequence of open sets around every point of the space. But this sequence may not be unique around each point. In Example \ref{2.2}, for $a\in X$, we have $a\in\{a\}\subseteq \{a,d\}\subseteq\{a,b,d\}\subseteq\{a,b,c,d\}=X$, is another sequence around $a$.
		}
	\end{example}
	
	Given a finite space $X$ and $x\in X$, we have the minimal open set $U_x$, as the intersection of all the open sets which contains $x$ (Ref.   \cite{Barmak}).
	
	\begin{definition}
		Let $X$ be a finite space. A nested sequence of open sets around $x\in X$ is a sequence $(U_j)_{j\geq 0}$ of open subsets $U_j$ of $X$ such that  $U_0=U_x$ and $U_j\subset U_{j+1},$ and there is no open set between $U_j$ and $U_{j+1}$ for all $j\geq 0$, where $U_j\subset U_{j+1}$ means $U_j$ is properly contained in $U_{j+1}$.
	\end{definition}
	
	Now, with the help of these nested sequences of open sets around a point $x\in X$, we define a notion of furtherness from $x$ to any point $y$ of $X$.
	
	\begin{definition}
		Let $X$ be a finite space and $x,y\in X$. The furtherness of $x$ from $y$ is the smallest whole number $k$ such that there exist a nested sequence $(U_j)_{j\geq0}$ of open sets around $x$ such that $y\in U_k$.
	\end{definition}

	In a finite space $X$, we always have finite number of nested sequences of open sets $(U_j)_{j\geq 0}$ around $x\in X$. We denote these sequences by $(U^i_j)_{j\geq 0}$, where $1\leq i\leq m, \text{ and } m\in \mathbb{N}$. For every $y\in X$ and a finite sequence $(U^i_j)_{j\geq 0}$, let $k_i$ be the least number such that $y\in U^i_{k_{i}}$. We define a map $\Psi_x: X\longrightarrow \{0,1,2,...,|X|-1\}\subseteq\mathbb{N}\cup \{0\}$, such that $\Psi_x(y)=\min\{k_i|1\leq i\leq m\}$.  Thus, $\Psi_{x}(y)$ is the furtherness of $x$ from $y$. It is clear that $\Psi_x(y)$ is uniquely determined for every $y\in X$. \par 
	
	This motivates us to introduce the furtherness function for every finite space $X$. 
	
	\begin{definition}[Furtherness function]
		Let $X$ be a finite space. The furtherness function on $X$ is a function $\Psi:X\times X\longrightarrow \{0,1,2,...,|X|-1\}\subseteq\mathbb{R}$, such that $\Psi(x,y)=\Psi_{x}(y),\forall x,y\in X$.
	\end{definition}

	Notice that $x\in U_0$, for every nested sequence $(U_j)_{j\geq 0}$ around $x$. So, the furtherness of every point from itself is zero. That is, $\Psi(x,x)=\Psi_{x}(x)=0,\forall x\in X.$\\
	
	Using the  partial order relation $\leq$ on a topology $\mathcal{T}$ of $X$ such that $U\leq V$ if and only if $U\subseteq V,$ for $U,V\in \mathcal{T},$ we can associate a poset daigram of $X$. For example




\begin{center}
	\begin{tikzpicture}  
		[scale=.9,auto=center,every node/.style={circle,fill=blue!10}] 
		
		\node (a1) at (5,4.5) [ball color=green, circle, draw=black, inner sep=0.1cm, label=left:X] {};  
		\node (a2) at (5,3)  [ball color=green, circle, draw=black, inner sep=0.1cm, label=left:{$\{1 ,2\}$}]{}; 
		\node (a3) at (5,1.5) [ball color=green, circle, draw=black, inner sep=0.1cm, label=below :$\emptyset$]{};  
		
		\draw (a1) -- (a2); 
		\draw (a2) -- (a3);  
		
		\tikzstyle{point1}=[ball color=cyan, circle, draw=black, inner sep=0.1cm]
		\tikzstyle{point2}=[ball color=green, circle, draw=black, inner sep=0.1cm]
		\tikzstyle{point3}=[ball color=red, circle, draw=black, inner sep=0.1cm]

		\node (v1) at (12,5) [ball color=cyan, circle, draw=black, inner sep=0.1cm, label=left:X] {};
		\node (v2) at (12,4) [point1, label=left:{$\{a,b,d\}$}] {};
		\node (v3) at (11,3) [point1,label=left:{$\{a,b\}$}] {};
		\node (v4) at (13,3) [point1,label=right:{$\{a,d\}$}] {};
		\node (v5) at (12,2) [point1,label=left:$\{a\}$] {};
		\node (v6) at (14,2) [point1,label=right: $\{d\}$] {};
		\node (v7) at (13,1) [point1, label=below:$\emptyset$] {};
		
		\draw (v1) -- (v2);
		\draw (v2) -- (v3) -- (v5);
		\draw (v2) -- (v4) -- (v5);
		\draw (v2) -- (v4) -- (v6);
		\draw (v5) -- (v7);
		\draw (v6) -- (v7);
		
		
	\end{tikzpicture}

\end{center}

\hspace{20mm}	Example \ref{2.1}\hspace{45mm} Example \ref{2.2}\\

In Example \ref{2.2}, furtherness of $a$ from $b$ is 1. As there are two nested sequences $a\in U^1_0=\{a\}\subset U^1_1= \{a,b\}\subset U^1_2=\{a,b,d\}\subset U^1_3=X$, and $a\in U^2_0=\{a\}\subset U^2_1= \{a,d\}\subset U^2_2=\{a,b,d\}\subset U^2_3=X$, of open sets around $a$. Here, $b\in  U^1_1$ and $b\in U^2_2$. So, $k_1=1$ and $k_2=2$. Thus, $\Psi(a,b)=\Psi_{a}(b)=\min\{k_1,k_2\}=1.$
Similarly, we have $\Psi(a,c)=3$ and $\Psi(a,d)=1.$\par 
In the same example, we can also notice that furtherness of $b$ from $a$ is 0. As there is only one nested sequence $b\in U_0=\{a,b\}\subset U_1=\{a,b,d\}\subset U_2=X$ around $b$, and $a\in U_0$. Thus, $\Psi(b,a)=\Psi_{b}(a)=0$. Similarly, the furtherness of $c$  from $a,b$ $\&$ $d$ is 0. It is cleat that the furthrerness function is not symmetric. This example also shows that $\Psi(x,y)=0$ does not imply $x=y$.


\begin{theorem}\label{2.6}
	
	Let $X$ be a finite $T_0$ space. Then $ \Psi(x,y)=\Psi(y,x)=0 \iff x=y,$ for $ x,y\in X$.
	
\end{theorem} 

\begin{proof}
	Let $x\neq y$. Now, $ \Psi(x,y)=0 \implies y\in U_x$ and $\Psi(y,x)=0 \implies x\in U_y$.
	This contradicts the fact that $X$ is $T_0$. So, $x=y$. Conversely, if $x=y$ then $\Psi(x,y)=\Psi(y,x)=0$.
\end{proof}

\begin{example}\label{2.7}
	{\normalfont
		If $X$ is a finite indiscrete space then the furtherness of every point is 0 from any other point. And if $X$ is a finite discrete space then the furtherness of every point is 1 from any other point. Thus, the furtherness function is a metric on finite discrete spaces.
		
	}

	

\end{example}

If $\Psi(x,y)=0$, then $U_y\subseteq U_x$. This gives the following result.

\begin{theorem}\label{2}
	Let $X$ be a finite topological space $x\in X$. Then the smallest open set containing $x$ is $U_x=\{y\in X |\Psi(x,y)=0\}$. 
\end{theorem}
By above theorem, we get that the singleton $\{a\}$ is open in a finite space $X$ $\iff$ $\Psi(a,b)\neq0,\forall~ b\in X\backslash\{a\}$.

\begin{remark}{\normalfont
		Given a finite topological space $X$ and $x,y\in X$, let $U_{x,y}$ be the smallest open set containing $\{x,y\}$. 
		It is easy to see that  $U_{x,y} = U_{x} \cup U_{y}$. 
		Similarly, for a subset $A$ of $X$, the smallest open set containing $A$ is $U_A= \bigcup\limits_{a\in A} U_a.$
	}
\end{remark}

Next, we observe that, if $\Psi(x,y)=k,$ for $x,y\in X$, then how open sets before $k^{th}$ term in nested sequence $(U_j)_{j\geq 0}$ around $x$ are related, and prove that its $k^{th}$ term is always  the same, whenever $y\in U_k$.

\begin{lemma}\label{2.8.}
	Let $X$ be a finite space and $(U_j)_{j\geq0}$ be a nested sequence of open sets around $x\in X$, then $U_j= U_{j-1}\cup U_a$, for some $a\in X\backslash U_{j-1}$.
\end{lemma}

\begin{proof}
	As $U_{j-1}\subset U_{j}$, let $a\in U_j\backslash U_{j-1}$. So, we get $U_a\subseteq U_j$, and $U_{j-1}\subset U_{j-1}\cup U_a \subseteq U_j$. By the definition of nested sequence of open sets, we get $U_{j}=(U_{j-1}\cup U_a)$.
\end{proof}

It is easy to observe that, for an open subset $A$ of a finite  space $X$ such that $U_x\subseteq A$ for some $x\in X$, there exists a nested sequence $(U_j)_{j\geq 0}$ of open subsets around $x$ such that $U_k= A$ for $k\in \mathbb{N}\cup \{0\}$.

\begin{theorem}\label{2.9.}
	Let $X$ be a finite space and $x,y\in X$ such that $\Psi(x,y)=k$. Then there exists a nested sequence $(U_j)_{j\geq 0}$ of open sets around $x$ such that $U_k=U_{x,y}$.
\end{theorem}
\begin{proof} 
	As  $\Psi(x,y)=k$, there exists a nested sequence $(V_j)_{j\geq 0}$ of open sets around $x$ such that $y\in V_k$. If $k=0$, then it is trivialy true because $y\in V_0=U_x\implies U_y\subseteq U_x=U_{x,y}=V_0$. Now let $k>0$. By Lemma \ref{2.8.}, we have $V_0= U_x$, $V_1= V_0\cup U_{a_1},$ for some $a_1\notin V_0$, $V_2= V_0\cup U_{a_1}\cup U_{a_2},$ for some  $a_2 \notin V_1$. By induction, we get $V_{k-1}= V_0\cup (\bigcup\limits_{i=1}^{k-1} U_{a_i})$, and $V_k= V_{k-1}\cup U_{a_k}$ for some $a_k\notin V_{k-1}$. As $y\notin V_{k-1}\implies y\in U_{a_k}$. So, $U_y\subseteq U_{a_k} \implies V_{k-1}\subset V_{k-1}\cup U_y\subseteq V_{k-1}\cup  U_{a_k} =V_k.$ 
	So, $V_k= V_{k-1}\cup U_y$. \par 
	As, $U_x$ is contained in open set $U_{x,y},$ there exists a nested sequence $(W_j)_{j\geq 0}$ around $x$ such that $W_p=U_{x,y}$ for some $p\in \mathbb{N}\cup\{0\}$. Now, $y\in W_p\implies p\geq k.$ Again by lemma \ref{2.8.}, we have $W_0= U_x$, and
	$W_p= U_x\cup (\bigcup\limits_{j=1}^p U_{b_j})$, for $b_j \notin W_{j-1}$.

	
	
	Now, $b_j\in  U_x \cup (\bigcup\limits_{j=1}^p U_{b_j})= U_x\cup U_y$ and $b_j\notin U_x \implies b_j\in U_y,  1\leq j\leq p$.
	So, $b_j\in V_k, 1\leq j\leq p.$ Next, we observe that each  $V_l$ contains atmost one $b_j$ which is not in $V_{l-1}, 1\leq l\leq k$. 
	If $b_j, b_{j'}\in V_l $ for some $1\leq l\leq k$ such that $b_j, b_{j'} \notin V_{l-1}$ and $b_j\neq b_{j'}$. Then $ V_{l-1}\subset  V_{l-1}\cup U_{b_j}\subset V_{l-1}\cup U_{b_j} \cup  U_{b_{j'}} \subseteq V_{l}$, which contradicts that there is no open set between $V_{l-1}$ and $V_l$. So, $V_k$ can have atmost $k$ distinct $b_j$ for $1\leq j\leq p$. But $b_j\in V_k,\forall~ 1\leq j\leq p.$ So, $p\leq k$.
	Thus, $p=k$. So, there exist a nested sequence $(U_j)_{j\geq 0}$ of open sets around $x$ such that $U_k=U_{x,y}$.
\end{proof}

Above result gives us an alternate definition of furtherness of $x\in X$ from any other point $y$ of $X$. That is $\Psi_x(y)$ is the least whole number $k$ such that there exist a nested sequence $(U_j)_{j\geq 0}$ of open subsets around $x$ such that $U_k=U_{x,y}$.\par 
If $x,y\in X$ such that $\Psi(x,y)=k$. Then, Theorem \ref{2.9.}, gives us the existence of a nested sequence $(U_j)_{j\geq1}$ around $x$ such that $U_k=U_{x,y}$. Next, we observe that this $U_k$ is unique in every nested sequence $(U_j)_{j\geq1}$ around $x$ such that $y\in U_k$.

\begin{corollary}
	Let $X$ be a finite space and $x,y\in X$ such that $\Psi(x,y)=k$. If $y\in U_k$ for a nested sequence of open sets $(U_j)_{j\geq 0}$ around $x$, then $U_k=U_{x,y}$.
\end{corollary}
\begin{proof}
	Let $(V_j)_{j\geq 0}$ be a nested sequence of open sets around $x$ such that $y\in V_k$. It is clear that $V_{k}= V_0\cup (\bigcup\limits_{i=1}^{k} U_{a_i})$, for some $a_i\notin V_{i-1}$. If $k=0$, then it is trivialy true because $y\in U_x=U_{x,y}=V_0$. Now, let $k>0$. By Theorem \ref{2.9.}, we have a nested sequence $(W_j)_{j\geq 0}$ around $x$ such that $W_k=U_{x,y}$, where  $W_k= U_x\cup (\bigcup\limits_{i=1}^k U_{b_i})=U_x\cup U_y$, for $b_i \notin W_{i-1}$. As $b_j\notin U_x$, $U_{b_j}\subseteq U_y,  1\leq j\leq p$. Therefore, $\bigcup\limits_{i=1}^p U_{b_i}\subseteq U_y$. As $k>0$ and $y\notin U_x$ and $y\in U_x\cup U_y =  U_x\cup (\bigcup\limits_{i=1}^p U_{b_i})$. So, $y\in \bigcup\limits_{i=1}^p U_{b_i}$. Therefore, $U_y\subseteq\bigcup\limits_{j=1}^p U_{b_j}$. So, $U_y=\bigcup\limits_{j=1}^p U_{b_j}$. Next, we show that $V_k=W_k=U_{x,y}$.\par 
	First, we observe that $\bigcup\limits_{i=1}^{k} U_{a_i}= \bigcup\limits_{i=1}^k U_{b_i}.$ We have  $\bigcup\limits_{i=1}^k U_{b_i}=U_y\subseteq U_{a_k}\subseteq \bigcup\limits_{i=1}^{k} U_{a_i}.$  As $a_k\in V_k=V_{k-1}\cup U_y$ and $a_k\notin V_{k-1}$,  we get $a_k\in U_y\implies U_{a_k}\subseteq U_y = \bigcup\limits_{i=1}^k U_{b_i}.$ Now, let there exist  $a_l$ such that $a_l\notin U_y$ for $1\leq l\leq (k-1).$ As $\Psi(x,y)=k, y\notin V_{k-1} \implies y\notin U_{a_i},$ for $1\leq i\leq (k-1)$. So, $y\notin U_{a_l}$. Consider, a nested sequence $(U_j)_{j\geq 0}$ of open sets around $x$, where $U_j =V_j$, for $j< l$ and $U_j= V_{j+1}\backslash U_{a_l}$ for $ j\geq l$. Here, $ y\in U_{k-1}= V_k\backslash U_{a_l}$. Thus $\Psi(x,y)<k$, a contradiction. Therefore, $a_l\in U_y, \forall~ 1\leq l\leq k$. So, $\bigcup\limits_{i=1}^{k} U_{a_i}\subseteq U_y= \bigcup\limits_{i=1}^k U_{b_i}.$ 
	Hence, $\bigcup\limits_{i=1}^{k} U_{a_i} = \bigcup\limits_{i=1}^{k} U_{b_i}\implies U_x\cup (\bigcup\limits_{i=1}^{k} U_{a_i} )= U_x\cup(\bigcup\limits_{i=1}^{k} U_{b_i}) \implies V_k=W_k=U_{x,y}.$
\end{proof}

In the next theorem, we show that the furtherness function of a finite  space also satisfies the triangle inequality like distance function in metric spaces.

\begin{theorem}\label{2.11.}
	Furtherness function on a finite space satisfies the triangle inequality.
\end{theorem}

\begin{proof} Let $X$ be a finite space. We show that $\Psi(x,y)\leq \Psi(x,z)+\Psi(z,y),\forall x,y,z \in X.$ If $\Psi(x,y)=0$, then there is nothing to prove. Let $\Psi(x,y)=p>0$ and $\Psi(x,z)=q$. If $p\leq q$, then we  are done.\\ 
	Now, let $p>q$. As 
	$\Psi(x,z)=q$, by Theorem \ref{2.9.}, there exists a nested sequence $(V_j)_{j\geq0}$ of open sets around $x$ such that $V_q=U_{x,z}$. 
	Now, let $\Psi(z,y)< p-q$. Then there exists a nested sequence $(W_j)_{j\geq 0}$ of open sets around $z$ such that $y\in W_k$ for some $k<p-q$. \par  
	Consider, a sequence $(U_j)_{j\geq 0}$ such that $U_j =V_j$, for $0\leq j\leq q$ and $U_j= U_x \cup W_{j-q}$ for $ j\geq q$. Clearly, $U_q
	=V_q$. As $(V_j)_{j\geq 0}$ and $(W_j)_{j\geq 0}$ are nested sequences, we get $(U_j)_{j\geq 0}$ ia a nested sequence around $x$. 
	We have $U_{q+k}=  U_x\cup W_k \implies y\in U_{q+k}.$ So, $\Psi(x,y)\leq q+k <q+p-q =p$, a conradiction. Therefore, $\Psi(z,y)\geq p-q \implies p\leq q+ \Psi(z,y)$. Thus, $\Psi(x,y)\leq \Psi(x,z)+\Psi(z,y)$.
\end{proof}

\begin{remark}\label{2.14.}
	{\normalfont
		By Theorem \ref{2.6} and Theorem \ref{2.11.}, we deduce that the furtherness function $\Psi$ defined on a finite $T_0$  space $X$ is an asymmetric distance on it. Thus,  $(X,\Psi)$ is an asymmetric metric space (Ref. \cite{Mennucci}). Moreover, the furtherness function on every finite space is a asymmetric semidistance (Ref. \cite{Mennucci}). 
	}
\end{remark}

In the following remarks, we make some observations about minimal open sets of finite $T_0$ spaces and the open sets in the nested sequence around any point of a finite $T_0$ spaces.

\begin{remark}\label{2.13.}{\normalfont
	If $X$ is finite $T_0$ space and $x\neq y \in X$, then $U_x\neq U_y$. By Theorem \ref{2.6},  if $x\neq y$, then either $\Psi(x,y)\neq 0$ or $\Psi(y,x)\neq 0$. That means either $x\notin U_y$ or $y\notin U_x$. Therefore, $U_x\neq U_y$.
	 }
\end{remark}

	\begin{remark}\label{2.13..}{\normalfont
     If $X$ is finite $T_0$ space , then for any nested sequence $(U_j)_{j\geq 0}$ of open sets around $x\in X$, we observe that $U_j$ conatins just one more point than $U_{j-1}$. By Lemma \ref{2.8.}, $U_j= U_{j-1}\cup U_a$, for some $a\in X\backslash U_{j-1}$. Let $p,q\in U_j$ such that $p,q\notin U_{j-1}$, then $p,q\in U_a$ and $p,q\notin U_{j-1}$. By Theorem \ref{2}, we get $\Psi(a,p)=\Psi(a,q)=0$, which implies $U_p\subset U_a$ and $U_q\subset U_a$. As $X$ is $T_0$, $U_p\neq U_q$ $\implies U_{j-1}\subset (U_{j-1}\cup U_{p}) \subset (U_{j-1}\cup U_{p} \cup U_q) \subseteq (U_{j-1}\cup U_a) = U_j,$ a contradiction to the nestedness of $(U_j)_{j\geq0}$. So, $U_j$ conatins just one more point than $U_{j-1}$.
}
\end{remark}

 The following result gives a characterization of furtherness for finite $T_0$ spaces in terms of cardinality of minimal open sets.
 
 \begin{corollary}\label{2.14}
 	Let $X$ be a finite $T_0$ space, then $\Psi(x,y)=|U_y\backslash U_x|, \forall x,y\in X.$
 \end{corollary}
 \begin{proof}
   Let $\Psi(x,y)=k$ for some $x,y\in X$. By Theorem \ref{2.9.}, we get a nested sequence $(U_j)_{j\geq 0}$ of open sets around $x$ such that $U_k=U_{x,y}$ and by Remark \ref{2.13..}, we get $|U_j\backslash U_{j-1}|=1, \forall j\geq 1.$ So, $|U_{x,y}|=|U_k|$. By induction, we get $|U_{k}|=|U_{x}|+k.$ Therefore, $\Psi(x,y)=|U_{x,y}\backslash U_{x}|= |U_y\backslash U_x|.$ 
 \end{proof}

\subsection{\textbf{Preorder on finite space}} 

Let $X$ be a finite topological space. We define a preorder (reflexive and transitive relation) on $X$ by $x\leq y $ if $ \Psi(y,x)=0$. By Theorem \ref{2.6}, if $X$ is $T_0$ then this preorder $\leq$ will be antisymmetric. So, for every finite $T_0$ space $X$, we have a partial order $\leq $ on $X$ defined as $x\leq y \iff \Psi(y,x)=0$. Therefore, we can associate a Poset diagram with every finite $T_0$ space. 

	
	
	%
	

The order defined above  can be seen as a characterization of the order introduced by Alexandroff (Ref. \cite{Alexandroff}) to explain  ono-to-one  corresepondence between finite $T_0$ spaces and finite partially ordered sets .\par 
In view of above characterization, the following results of R. E. Stong (Ref. \cite{Stong}) can be restated in terms of furtherness function. 


\begin{proposition}\label{3.8.}
	A function $f:X\longrightarrow Y$ between finite topological spaces is continuous if and only if $\Psi_X(x,y)=0\implies \Psi_Y(f(x),f(y))=0,\forall~ x,y\in X$.
\end{proposition}


\begin{proposition}
	Let $X$ be a topological space and $Y$ be a finite space. If $f,g: X\longrightarrow Y$ are two continuous maps such that $\Psi(f(x),g(x))=0,\forall~ x\in X$, then $f$ and $g$ are homotopic.
\end{proposition}

\begin{proposition}
	A finite space $X$ is contractible if there exists $x\in X$ such that $\Psi(x,y)=0,\forall~ y\in X$ or $\Psi(y,x)=0,\forall ~ y\in X$.
\end{proposition}

\begin{proposition}\label{3..8}
	Let $X$ be a finite space and define a relation on $X$ by $x\sim y$ if $\Psi(x,y)=\Psi(y,x)=0$. Then, the quotient $X/$$\sim$ is $T_0$ and  the quotient map is a homotopy equivalence.
\end{proposition}


Let $\pi:X\longrightarrow X/$$\sim$ be the natural map such that $\pi(x)=[x]$ for $x\in X$. As $\pi$ is continuous, by Proposition \ref{3.8.}, we get $\Psi_X(x,y)=0\implies \Psi_{X/\sim}([x],[y])=0,\forall~ x,y\in X$. Thus, $x\leq y \implies [x]\leq [y],\forall~ x,y\in X$. By Theorem \ref{2}, we get $U_{[x]}=\{[y]|\Psi(x,y)=0\}.$
In the next result, we observe that the furtherness between two equivalence classes of $X/$$\sim$  is the same as furtherness between corresponding elements. 

\begin{theorem}\label{3..9}
	Let $X$ be a finite space and $\sim$ be a equivalence relation on $X$ such that $x\sim y$ if $\Psi_X(x,y)=\Psi_X(y,x)=0$. Then $\Psi_{X/\sim}([x],[y])=\Psi_{X}(x,y)$ for all $x,y\in X$.
\end{theorem}

\begin{proof} Let $\Psi_X(x,y)=k$. Then there exists a nested sequence $(V_j)_{j\geq 0}$ of open sets around $x$ such that $y\in V_k$. By Lemma \ref{2.8.}, $V_j= V_{j-1}\cup U_{a_j}$, for some $a_j\in X\backslash V_{j-1}$. So, $y\in U_{a_k}$ and $y\notin U_{a_l}$ for $l<k$.
	For $[a]\in X/$$\sim,$ we have  $U_{[a]}=\{[b]|\Psi_X(a,b)=0\}$. Now, we construct a nested sequence $(W_j)_{j\geq0}$ of open sets around $[x]$ such that $W_{j} = W_{j-1}\cup U_{[a_{j}]}$. As $[y]\in U_{[a_k]}$, $[y]\in W_k$. Therefore, $\Psi_{X/\sim}([x],[y])\leq k.$\par 
	If $\Psi_{X/\sim}([x],[y])< k,$ then there exist a nested sequence $(U_j)_{j\geq0}$ of open sets around $[x]$ such that $[y]\in U_m$ for some $m<k$. Again by Lemma \ref{2.8.}, we have $U_j= U_{j-1}\cup U_{[b_j]}$, for some $b_j\in X\backslash U_{j-1}$. So, $[y]\in U_{[b_m]} \implies y\leq b_m \implies y\in U_{b_m}$. Now, consider a nested sequence $(U'_j)_{j\geq0}$ of open sets around $x$ such that $U'_j= U'_{j-1}\cup U_{b_j}$. So, $y\in U_m$, and hence $\Psi(x,y)\leq m<k$, a contradiction. Therefore, $\Psi_{X/\sim}([x],[y])=\Psi_{X}(x,y)$ for all $x,y\in X$.
\end{proof}

Using  Corollary \ref{2.14} and Theorem \ref{3..9}, we get a characterization of furtherness in term of cardinality of minimal open sets of its quotient space $X/$$\sim$ defined in Propostion \ref{3..8}.

\begin{corollary}
	Let $X$ be a finite space and $\sim$ be a equivalence relation on $X$ such that $x\sim y$ if $\Psi_X(x,y)=\Psi_X(y,x)=0$. Then $\Psi(x,y)= |U_{[y]}\backslash U_{[x]}|,\forall x,y\in X$.
\end{corollary}

We can identify beat points \cite{Barmak} of a space from the furtherness between its points.

\begin{remark}\label{3.9}
	{\normalfont
		Let $X$ be a finite space. Then $x\in X$ is a down beat point of $X$ if there exists exactly one $y\in X$ such that $\Psi(x,y)=0$ and there does not exist any $z\in X$ such that $\Psi(x,z)$ and $\Psi(z,y)$ are both zero at the same time.\par 
		Similarly, $x\in X$ is a up beat point of $X$ if there exists exactly one $y\in X$ such that $\Psi(y,x)=0$ and there does not exist any $z\in X$ such that $\Psi(y,z)$ and $\Psi(z,x)$ are both zero at the same time. Recall that a finite $T_0$ space with no beat points is called  minimal finitie space (Ref. \cite{Barmak}).
	}
\end{remark}

\begin{proposition}
	Let  $f$  be a continuous map from a minimal finite space $X$ to itself. Then $\Psi(f(x),x)=0,\forall~ x\in X\implies f=1_X$.
\end{proposition}

\subsection{\textbf{Furtherness function on product of finite spaces}}

Let $X$ and $Y$ be two finite spaces.
In this section, we observe that how the furtherness function on the product space $X\times Y$ is related with the furtherness functions on $X$ and $Y$.\\

\begin{definition}[Furtherness preserving map] A furtherness preserving map between two finite topological spaces $X$ and $Y$ is a funtion $f:X\longrightarrow Y$ such that $\Psi_{X}(a,b)=\Psi_{Y}(f(a),f(b)),\forall a,b\in X$.
	
\end{definition} 

By Proposition \ref{3.8.}, we get that every furtherness preserving map is continuous. Notice that any map between two trivial spaces is always  a furtherness preserving map, and any injective map between two discrete spaces preserves the furtherness. Moreover, any homeomorphism between two finite spaces is a furtherness preserving map.

Using Corollary \ref{2.14} along with set-theoretic identities,we obtain the furtherness on the product of finite $T_0$ spaces.

\begin{theorem}\label{4.1}
	Let $X\times Y$ be the product of two finite $T_0$ spaces $X$ and $Y$. Then $\Psi_{X\times Y}((a,b),(c,d))=\Psi_{X}(a,c)|U_d|+\Psi_{Y}(b,d)|U_c|-\Psi_{X}(a,c)\Psi_{Y}(b,d),$ is the furtherness on $X\times Y$, where  $(a,b),(c,d)\in X\times Y,$ 
\end{theorem}

\begin{example}
Let $X=\{a,b\}$ and $Y=\{x,y\}$ be two spsces, where $\tau_X=\{\emptyset,\{a\},X\}, \&~$ $\tau_Y=\{\emptyset,\{x\},Y\}$ are the topologies on $X$ and $Y$, repectively. It is easy to observe that, $\Psi_X(a,b)=1$, $\Psi_Y(x,y)=1$ and $\Psi_{X\times Y}((a,x),(b,y))=3$.
\end{example}


Let $X/$$\sim$ and $Y/$$\sim$ be quotient spaces of two finite spaces $X$ and $Y$, respectively, as defined in Proposition \ref{3..8}. Then both  $X/$$\sim$ $\&$ $Y/$$\sim$ are $T_0$ spaces, and so is their product  $X/$$\sim$ $\times Y/$$\sim$. Assume that $x_1,X_2\in X$ and $y_1,y_2\in Y$. Then, $x_1\sim x_2$ $\&$ $ y_1\sim y_2$ $\iff$ $\Psi_X(x_1,x_2)=\Psi_X(x_2,x_1)=0$ and $\Psi_Y(y_1,y_2)=\Psi_Y(y_2,y_1)=0 \iff U_{x_1}=U_{x_2}$ and $U_{y_1}=U_{y_2}$ $\iff$ $U_{x_1}\times U_{y_1} = U_{x_2}\times U_{y_2} \iff U_{(x_1,y_1)}=U_{(x_2,y_2)}$ $\iff$ $\Psi_{X\times Y}((x_1,y_1),(x_2,y_2))=\Psi_{X\times Y}((x_2,y_2),(x_1,y_1))=0$ $\iff$ $(x_1,y_1)\sim (x_2,y_2)$. This implies that the quotient space $(X\times Y)/$$\sim$ is homeomorphic to the product  $X/$$\sim$ $\times Y/$$\sim$. And, a map  $f:(X\times Y)/$$\sim~ \longrightarrow  X/$$\sim \times Y/$$\sim$ defined by  $f([x,y])=([x],[y])$ is a homeomorphism. So, $\Psi_{(X\times Y)/\sim}([(a,b)],[(c,d)])=\Psi_{X/\sim \times Y/\sim}(([a],[b]),([c],[d]))$.

Now, using above observation, we give the furtherness function on the product of two finite spaces.

\begin{theorem}\label{4.1.}
	Let $X\times Y$ be the product of two finite spaces $X$ and $Y$. Then $\Psi_{X\times Y}((a,b),(c,d))=\Psi_{X}(a,c)|U_{[d]}|+\Psi_{Y}(b,d)|U_{[c]}|-\Psi_{X}(a,c)\Psi_{Y}(b,d),$ is the furtherness on $X\times Y$, where $(a,b),(c,d)\in X\times Y,$ 
\end{theorem}

\begin{proof}
	Let $(a,b),(c,d)\in X\times Y.$ By Theorem \ref{3..9}, $\Psi_{X\times Y}((a,b),(c,d))=\Psi_{(X\times Y)/\sim}\\
	([(a,b)],[(c,d)])$. As $(X\times Y)/$$\sim$ is homeomorphic to the product  $X/$$\sim$ $\times Y/$$\sim$, we get $\Psi_{(X\times Y)/\sim}([(a,b)],[(c,d)])=\Psi_{X/\sim \times Y/\sim}(([a],[b]),([c],[d]))$. As  $X/$$\sim$ $\times Y/$$\sim$ is $T_0$, by Theorem \ref{4.1},  $\Psi_{X/\sim \times Y/\sim}(([a],[b]),([c],[d]))= \Psi_{X/\sim}([a],[c])|U_{[d]}|+\Psi_{Y/\sim}([b],[d])$
	$|U_{[c]}|-\Psi_{X/\sim}([a],[c])\Psi_{ Y/\sim}([b],[d]).$ Again, by using Theorem \ref{3..9}, we get $\Psi_{X\times Y}((a,b),\\
	(c,d))=\Psi_{X}(a,c)|U_{[d]}|+\Psi_{Y}(b,d)|U_{[c]}|-\Psi_{X}(a,c)\Psi_{Y}(b,d).$
\end{proof}

Let $X_i/$$\sim$ be quotient space of finite spaces $X_i, 1\leq i\leq n$. Then $(\prod_{i=1}^n X_i)/$$\sim$ is homeomorphic to $\prod_{i=1}^n (X_i/$$\sim)$, where $x\sim y$ in  $\prod_{i=1}^n X_i \iff x_i \sim y_i, \forall~ 1\leq i\leq n$, for  $x=(x_1,x_2,...,x_n),y=(y_1,y_2,...,y_n)\in  \prod_{i=1}^n X_i$.
Using induction, we obtain furtherness for the product of finite spaces.

\begin{corollary}   
	
	Let $X$ be the product of finite spaces $X_i$, $1\leq i\leq n$, for $n\in \mathbb{N}$. Then $\Psi_{X}(a,b)= |U_{[b]}\backslash U_{[a]}|= |(U_{[b_1]}\times U_{[b_2]}\times ... U_{[b_n]})\backslash(U_{[a_1]}\times U_{[a_2]}\times ... U_{[a_n]})|$ is the furtherness on $X$, where $a=(a_1,a_2,...,a_n)$ and $b=(b_1,b_2,...,b_n)\in X.$ 
\end{corollary}

\section{\textbf{Topology induced by furtherness}}

By Remark \ref{2.14.}, the furtherness function $\Psi$ on a finite space $X$ is an asymmetric semidistance. For each $x\in X$ and $n\in \mathbb{N}$,  $B^+(x,n)=\{y\in X|\Psi(x,y)<n\}$ and $B^-(x,n)=\{y\in X|\Psi(y,x)<n\}$ are the forward and the backward open balls of radius $n$ and centered at $x$, respectively (Ref. \cite{Mennucci}).

\subsection{\textbf{Topology generated by balls}}

Let $\mathcal{B}^+$ be the collection of forward open balls of $(X,\Psi)$. Clearly, $X=\bigcup\{B^+(x,n)
|x\in X, n\in \mathbb{N}\}$, and if $x\in B^+(a,n_1)\cap B^+(b,n_2)$ for some $a,b\in X$ and $n_1,n_2\in\mathbb{N}$, then $x\in B^+(x,1)\subseteq B^+(a,n_1)\cap B^+(b,n_2).$ So, the family $\mathcal{B}^+$ forms a basis for a topology on $X$. For each $x\in X,$ we get $U_x=B^+(x,1).$ This gives the following result. 

\begin{theorem}
	The collection 
	of forward open balls in a finite space induces the underlying topology of the space.
\end{theorem}



Recall that the set of closed subsets of a finite space $X$ with topology $\mathcal{T}$ is also a topology on the underlying set $X$. The set $X$ with this topology is 
called the opposite topology of $\mathcal{T}$ and it is denoted by $X^{op}$. Thus, $X^{op}$ represents the same underlying set $X$ with the topology where the open sets are precisely the closed sets of $\mathcal{T}$ (Ref. \cite{Barmak}).\par 
It is easy to observe that the closure of $A\subseteq X$ in  $X$ is equal to $\{y\in X|\Psi(y,A)=0\}$. Let $U_x^{op}$ be the minimal open subset of $X^{op}$ containing $x$. Thus,  $U_x^{op}$ is the smallest closed subset of $X$ containg $x$. 
Therefore, the closure of $\{x\}$ in  $X$ is $U_x^{op}=
\{y\in X|\Psi(y,x)=0\}=B^-(x,1),\forall x\in X$. 



Let $\mathcal{B}^-$ be the collection of backward open balls of $(X,\Psi)$.
Similar to $\mathcal{B}^+$ there is a topology on $X$ for which $\mathcal{B}^-$ is a basis. 

\begin{theorem}
	The collection 
	of backward open balls in a finite space induces the opposite topology of the space.
\end{theorem}

\begin{remark}\label{3.3.}{\normalfont
	For a finite topological space $X$, the function $\overline{\Psi}:X\times X\longrightarrow \mathbb{N}\cup\{0\}$ defined by $\overline{\Psi}(x,y)=\max\{\Psi(x,y),\Psi(y,x)\}$ is a pseudo-metric.}
	\end{remark}
	
	In the pseudo-metric space  $(X,\overline{\Psi})$, for each $ x\in X$ and $n\in\mathbb{N}$, the open ball $ B(x,n)=\{y\in X|\overline{\Psi}(x,y)<n\}=B^+(x,n)\cap B^-(x,n)$. Let $\mathcal{B}$ be the collection of open balls of pseudo-metric space $(X,\overline{\Psi})$. Let $\overline{\mathcal{T}}$ be the topology generated by $\mathcal{B}$ on $X$. 
	It is clear that the topology induced by $\mathcal{B}$ on $X$ is finer than both $\mathcal{T}$ and $\mathcal{T}^{op}$. Infact, $\overline{\mathcal{T}}$ is the smallest topology on $X$ containing $\mathcal{T}\cup \mathcal{T}^{op}$. Also notice that $(X,\overline{\mathcal{T}})$ is always a disconnected space, if $|X|>1$. \par 

	\begin{remark}{\normalfont
	If a finite topological space $X$ is $T_0$ then by Theorem  \ref{2.6}, $\overline{\Psi}$ is a metric on $X$, and hence the topology induced by $\overline{\Psi}$ is the discrete topology on $X$\par 
	
}
\end{remark}

\subsection{\textbf{Furtherness matrix}}

We can associate a matrix with every finite topological space with the help of furtherness function.

\begin{definition}
Let $X$ be a finite space with $n$ elements $\{a_i|1\leq i\leq n\}$. Then the furtherness matrix associated with  $X$ is an $n\times n$ matrix, where $(i,j)^{th}$ entry of furtherness matrix is $\Psi(a_i,a_j)$, for $a_i,a_j\in X$.

\end{definition}

We denote this matrix by $\Psi(X)$. So, $\Psi(X)_{i,j}=\Psi(a_i,a_j)=|U_{[a_{j}]}\backslash U_{[a_{i}]}|, \forall i,j\in X$. In particular, if $X$ if $T_0$, then $\Psi(X)_{i,j}=|U_{a_{j}}\backslash U_{a_{i}}|, \forall i,j\in X$. \par

In Example \ref{2.2}, the furtherness matrix  $\Psi(X)$ is 
$\bordermatrix{ & a & b & c & d \cr
a &	0 & 1 & 3 & 1\cr
b & 0 & 0 & 2 & 1\cr
c &	0 & 0 & 0 & 0\cr
d &	1 & 2 & 3 & 0 }$.\\
 
Notice that the diagonal of a furtherness matrix is always zero. The	furtherness matrix of a finite space consists of many information about the space.
\begin{itemize}
\item[1.] The furtherness matrix of an indiscrete space of $n$ elements is an $n\times n$ zero matrix.
\item[2.] The furtherness matrix of a discrete space of $n$ elements is an $n\times n$ matrix whose nondiagonal entries are one.
\item[3.] By Theorem \ref{2}, we get that singleton $\{x\}$ is open in a finite space $X$ if and only if the row corresponding to $x$ in $\Psi(X)$ has just one entry as zero.
\item[4.]  The minimal open set $U_x$ containg $x$ in $X$ consists all those $y\in X$ such that the column of $y$ has a zero in the row corresponding to $x$. Moreover, $|U_x|$ is same as the number of zeros in the row corresponding to $x$.
\item[5.]   The smallest closed set $\overline{\{x\}}$ containg $x$ in $X$ consists all those $y\in X$ such that the row of $y$ has a zero in the column corresponding to $x$. Moreover, the cardinality of $\overline{\{x\}}$ is same as the number of zeros in the column corresponding to $x$.
\end{itemize}



We further establish some results using furtherness matrix.

\begin{theorem}
	Let $X$ be a finite space. The furtherness of a point $x$ is zero from a point $y$ if and only if  every entry in the row of $x$ is less than or equal to the corresponding entry in the row of $y$.
\end{theorem}  

\begin{proof}
	Let  $\Psi(x,y)=0$ for some $x,y\in X$. We show that $\Psi(x,a)\leq \Psi(y,a),\forall a\in X$. Let $\Psi(x,a)=k$. 
	Now, if $\Psi(y,a)<k$ then there exists a nested sequence $(V_j)_{j\geq 0}$ of open sets  around $y$ such that $a\in V_l$, for some $l<k$. Consider a sequence $(W_j)_{j\geq 0}$ such that $W_j= U_x \cup V_j$. 
	Clearly, $W_j\subseteq W_{j+1},\forall j$. If $W_j\subset W_{j+1},\forall j,$ then $(W_j)_{j\geq 0}$ is a nested sequence of open sets around $x$. As $x\in W_l,\Psi(a,x)<k,$ a contradiction.  If $W_j= W_{j+1}$ for some $j$ then again we get a nested sequence around $x$ such that $\Psi(x,a)$ is much less than $k$, a contraciction. So, $\Psi(y,a)\geq k$. Thus, $\Psi(x,a)\leq \Psi(y,a),\forall a\in X$. Conversely, let $\Psi(x,a)\leq \Psi(y,a),\forall a\in X$, then $\Psi(x,y)\leq \Psi(y,y)=0\implies \Psi(x,y)=0.$
\end{proof}  

\begin{theorem}
	A finite space $X$ is $T_0$ if and only if all the rows of $\Psi(X)$ are distinct.
\end{theorem}		

\begin{proof}
	Let $X$ be a $T_0$ space. Assume that for some $x\neq y$ in $X$, the rows corresponding to $x$ and $y$ are identical. Thus,  $\Psi(x,z)=\Psi(y,z),\forall ~ z\in X$. As $x\in X$, $\Psi(y,x)=\Psi(x,x)=0$. Similarly, $\Psi(x,y)=\Psi(y,y)=0$. Since $X$ is $T_0$, by Theorem \ref{2.6}, we get $x=y$, a contradiction. So, 
	all the rows of $\Psi(X)$ are distinct.\par 
	If $X$ is not $T_0$ then for some $x\neq y$, we get $U_x=U_y$. 
	So, the furtherness of $x$ from any point $z\in X$ is same as the furtherness of $y$ from $z$. Thus, the row corresponding to $x$ and $y$ are  identical, a contradiction. Hence, our claim. 
\end{proof}		

Similarly, the above result is also true for the columns of furtherness matrix. 
\begin{theorem}
	A finite space $X$ is $T_0$ if and only if all the columns of $\Psi(X)$ are distinct.
\end{theorem}

Example \ref{2.1}, the furtherness matrix  $\Psi(X)$ is 
$\bordermatrix{ & 1 & 2 & 3  \cr
	1 &	0 & 0 & 1 \cr
	2 & 0 & 0 & 1 \cr
	3 &	0 & 0 & 0 }$. Here, the rows corresponding to 1 and 2 are the same. So, $X$ is not $T_0$.


We know that finite $T_0$ spaces are posets. Next, using furtherness matrix, we identify maximum and minimum points \cite{Barmak} of a finite $T_0$ space. It is easy to observe the following result.

\begin{theorem}
	Let $X$ be a finite $T_0$ space. Then $x\in X$ is the maximum point of $X$ if and only if the row corresponding to $x$ in $\Psi(X)$ is zero. Also, $x\in X$ is the minimum point of $X$ if and only if the column corresponding to $x$ in $\Psi(X)$ is zero.
\end{theorem}

\begin{theorem}
		A finite space is contractible if its furtherness matrix has a zero row or a zero  column.
\end{theorem}

\begin{theorem}
	Let $X$ be a finite space and $x\in X$. Then the furtherness of $y\in X$ from $x$ can not be more than the  number of zeroes in the row corresponding to $x$ in the furtherness matrix.
\end{theorem}
\begin{proof}
	Suppose that the row corresponding to $x$ in the furtherness matrix has $k$ zeroes. We show that $\Psi(y,x)\leq k,$ for all $y\in X$. Let  $a_1,a_2,...,a_k$ be $k$ elements in $X$ such that $\Psi(x,a_i)=0$. Therefore, $U_{a_i}\subseteq U_x$ for $1\leq i\leq k$. Let $A$ be the collection of all the $a_i$ such that there $U_{a_i}$'s are distinct. Clearly, $|A|=p\leq k$. For $y\in X$, consider a nested sequence $(U_j)_{j\geq0}$ of open sets around $y$ in which $U_0=U_y$ and $U_j=U_{j-1}\cup U_{a_j}$ for $1\leq j\leq p$. As $U_x=\{a\in X|\Psi(x,a)=0\},$ we get $x\in U_p$, and $x\notin U_j$, for $j<p$.
	Thus, $\Psi(y,x)\leq p \leq k$.
\end{proof}

\begin{corollary}
	The furtherness between two points of a finite space $X$ can not be more than the maximum number of zeroes in any row of its furtherness matrix.
\end{corollary}

	\section{\textbf{Center and Radius of a subset of finite topologiocal space}}
	
In this section, first we define the furtherness of a point  from any subset of a finite space $X$. Moreover, the furtherness of a subset from any  subset of $X$ can also be determined. 
	
	\begin{definition}
		Let $B$ be a nonempty subset of a fnite space $X.$ For $a\in X$,  the furtherness of $a$ from $B$ is defined by  $\Psi(a,B)=\Psi_{a}(B)=\min\limits_{b\in B}\Psi_{a}(b)=\min\limits_{b\in B}\Psi(a,b)$. For a nonempty subset $A$ of $X$, the furtherness of $A$ from $B$ is defined by  $\Psi(A,B)=\Psi_{A}(B)=\min\limits_{a\in A}\Psi(a,B)=\min\limits_{a\in A}\{\min\limits_{b\in B}\Psi(a,b)\}$.\par 
		{\normalfont We have $\Psi(a,b)$ is the least $k$ such that $b\in U_{k}$ for a finite nested sequence  $(U_j)_{j\geq0}$ of open sets around $a$ and $\Psi(a,B)$ is the  least $k$ such that $b\in U_{k}$ for some $b\in B$. Thus $\Psi(a,B)$ is the least number  $k$ such that $B\cap   U_{k}\neq\emptyset$. So, if $B= \emptyset$ then, $\Psi(A,B)=\infty$. \par
			If $A= \emptyset$ then, $\Psi(A,B)=\infty$, because if  $\Psi(A,B)=k\in \mathbb{N}\cup\{0\}$, then $\Psi(a,b)=k,$ for some $a\in A$ and $b\in B$, a contradiction as there does  ot exist any $a\in A$. So, $\Psi(A,B)=\infty$.
		}

	\end{definition}
	

	In Example \ref{2.2}, the furtherness of $a$ from $\{b,c\}$  is $\Psi(a,\{b,c\})=\min\{\Psi_{a}(b),\Psi_{a}(c)\}\\ =\min\{1,3\}=1,$ whereas the  furtherness of $\{b,c\}$ from  $a$ is $\Psi(\{b,c\},a)=\min\{\Psi_{b}(a),\\
	\Psi_{c}(a)\} =\min\{0,0\}=0.$\par 
	
	Next, if $A$ and $B$ are subsets of a finite space $X$ such that $A\cap B\neq \emptyset$, then $\Psi(A,B)=\Psi(B,A)=0$. But its converse is no true. As in Example \ref{2.2}, $\Psi(\{b,c\},a)=0$ but $\{b,c\}\cap \{a\}= \emptyset$.
	
	
	Notice that, if  $X$ is a finite space and $A\subseteq B$ are subsets of $X$, then $\Psi(x,A)\geq \Psi(x,B)$ and $\Psi(A,x)\geq \Psi(B,x), \forall x\in X$.
	
	\begin{theorem}
	Let $A$ be a subset of finite topological space $X$. Then $\Psi(a,A)=\Psi(a,\overline{A}),\forall a\in X$.
	\end{theorem}
	\begin{proof}
	As $A\subseteq \overline{A}$,  $\Psi(a,A)\geq \Psi(a,\overline{A}),\forall a\in X$. By definition of furtherness $\Psi(a,\overline{A})=\min\limits_{b\in\overline{A}}\Psi(a,b)=\Psi(a,b')$ for some $b'\in\overline{A}$. If $b'\in A$, then $\Psi(a,A)=\Psi(a,\overline{A})$. If $b'\notin A$, then $b'\in \partial_{X}(A)$. This means $U_{b'}\cap A\neq\emptyset\implies \Psi(b',a')=0,$ for some $a'\in A \implies \Psi(b',A)=0$. Now as furtherness satisfies triangle inequality, we get $\Psi(a,A)\leq \Psi(a,a')\leq \Psi(a,b')+\Psi(b',a')= \Psi(a,\overline{A}).$ Hence, $\Psi(a,A)=\Psi(a,\overline{A}),\forall a\in X$.
	\end{proof}

	
	
	\begin{remark}{\normalfont
			If $U_x=U_y$, 
			then $\Psi(x,z)=\Psi(y,z)$ and $\Psi(z,x)=\Psi(z,y),\forall x,y,z\in X.$ Moreover, $\Psi(x,A)=\Psi(y,A)$ and $\Psi(A,x)=\Psi(A,y),\forall A\subseteq X.$}
	\end{remark}

	\begin{theorem}
		Let $X$ be a finite space and $A, B$ are subsets of $X$ such that $\Psi(A,B)=0$, then 
		there does not exist any disjoint open sets $U$ and $V$ such that $A\subseteq U$ and $B\subseteq V$.
	\end{theorem}
	\begin{proof}
		As $\Psi(A,B)=0$, there exists $a\in A$ and $b\in B$ such that $\Psi(a,b)=0$. Thus, $b\in U_a$. Which means every open neighbourhood of $a$ contains $b$. So, there does not exist two disjoint open sers $U$ and $V$ such that $A\subseteq U$ and $B\subseteq V$.
	\end{proof}

	\subsection{Center and Radius of a subset of finite space} 
In a metric space $X$, for a subset $A$ of $X$, the center $Cent_{X}(A)$ of $A$ consists of all those points of $A$, which are at the maximum distance from the boundary of $A$, see \cite{Badra}. Now, we have a distance like notion of furtherness for finite spaces. Using that,  we define the center and radius of a subset $A$ of finite  space $ X$.

\begin{definition}[Center of a subset] 
	The center of subset $A$ of a finite topological space $X$ is the set
	$ \{a \in A | \Psi(a,\partial_{X}(A)) \geq \Psi(b,\partial_{X}(A)), \forall\, b \in A\}$, where $\partial_{X}(A)$ is the boundary of $A$ in $X$. We denote the center of $A$ in $X$ by $Cent_{X}(A).$

	
\end{definition}

Thus the center of A is the set of all those elements of $A$ which are at the maximum furtherness from the boundary of $A$. 

\begin{definition}[Radius of a subset]
	The radius of subset $A$ of a finite topological space $X$ is the furtherness of the center of $A$ from the boundary of $A$.  We denote the radius of A in X by $rad_{X}(A).$ 
\end{definition}

It is clear that for every point in $Cent_{X}(A)$ has the same furtherness from $\partial_{X}(A)$. We have $rad_{X}(A)= \Psi(Cent_{X}(A),\partial_{X}(A)) = \Psi(a, \partial_{X}(A)) , \forall a\in Cent_{X}(A)$.\par 
Notice that for a nonempty subset $A$ of $X$, its center is always nonempty. For any clopen subset $A$ of a finite space $X$, $Cent_{X}(A)= A$ and $rad_{{X}}(A) = \infty$.
Note that the radius of a subset of finite space is infinite if and only if it is clopen. \vspace{3mm}


In Example \ref{2.1}, for $A=\{2,3\}\subseteq X$, we have $\partial_{X}(A)=\{1,2,3\}$. So, $Cent_{X}(A)=\{2,3\}$ and $rad_{X}(A)=0$.\par 
In Example \ref{2.2},  for $A=\{a,c\}\subseteq X$, we have $\partial_{X}(A)=\{b,c\}$. So, $Cent_{X}(A)=\{a\}$ and $rad_{X}(A)=1$.\par 

Notice that if $A\subseteq\partial_{X}(A),$ then 
$\Psi(a,\partial_{X}(A))=0,\forall a\in A$. 
So,  we have the following result.

\begin{theorem}
Let $A$ be a nonempty subset of a finite space $X$ such that $A^\circ=\emptyset$, then $Cent_{X}(A)=A$ and $rad_{X}(A)=0$.
\end{theorem}

If interior of $A$ is nonempty then $\Psi_X(a,\partial_X(A))>0$ for all $a\in A^\circ$, and we have

\begin{theorem}
Let $A$ be a subset of a finite  space $X$ such that $A^\circ\neq\emptyset$, then $Cent_{X}(A)\subseteq A^\circ$ and $rad_{X}(A)>0$.
\end{theorem}

For a nonempty subset $A$ of a finite  space $X$, we get  $A^\circ=\emptyset \iff rad_{X}(A)=0$.


\begin{corollary}
	Radius of every open subset of a finite space is always positive.
\end{corollary}

\begin{remark}
	{\normalfont
		Note that the furtherness map $\Psi$ of a finite space $X$ is not always a continuous map, where $\Psi(X)$ has the subspace topology induced from Euclidean line $\mathbb{R}$. So, unlike metric spaces, $Cent_{X}(A)$ may not be a closed subset of $A$, for $A\subseteq X$. In Example \ref{2.2}, for $A=\{a,c\}\subseteq X,$ we get $Cent_{X}(A)=\{a\},$ which is not closed in $A$.}
\end{remark}

Next, we have a relationship between
the radii of $A^\circ$ and $\overline{A}$ with the radius of A. As the  boundaries of $A^\circ$ and $\overline{A}$ are contained in the boundary of $A$, so 
we have the following result.

\begin{theorem}
	Let $X$ be a finite  space and $A \subseteq X$.
	Then $rad_X(A) \leq rad_X(A^\circ)$ and $rad_X(A) \leq  rad_X(\overline{A})$.
\end{theorem}


Notice that if $X$ is a finite space and $A \subseteq Y \subseteq X$, then $rad_X(A) \leq  rad_Y (A).$\\

	By Theorem \ref{4.1.}, we have the notion of furtherness for the product of finite spaces, using that one can find the center and radius of subsets of the product of finite spaces.

\subsection{\textbf{Center and radius of union of subsets of a finite space}}

In this section, we derive $Cent_{X}(A \cup B)$ and  $rad_{X}(A \cup B)$ for nonclopen subsets $A$ and $B$ within a finite space $X$. Given nonclopen subsets $A$ and $B$ of a finite space $X$, let\par
\hspace{2mm} $\Tilde{A}=\{a\in Cent_{X}(A)|\Psi(a, \partial_{X}(B))<rad_{X}(A)\}$, and\par
\hspace{2mm} $\Tilde{B}=\{b\in Cent_{X}(B)|\Psi(b,\partial_{X}(A))<rad_{X}(B)\}.$\\
Using these notations, we determine $Cent_{X}(A \cup B)$ and  $rad_{X}(A \cup B)$ for non-clopen subsets A and B in a finite  space X.

\begin{theorem}\label{2.25}
Let A and B be  nonclopen separated subsets of a finite space $X.$ Then, $rad_{X}(A\cup B)\leq \max\{rad_{X}(A),rad_{X}(B)\}.$  Moreover,
\begin{itemize}
	\item[(i)] if $rad_{X}(A)>rad_{X}(B)$ and $Cent_{X}(A)\backslash\Tilde{A}\neq\emptyset,$ then $Cent_{X}(A\cup B)=Cent_{X}(A)\backslash\Tilde{A}$ $\&$ $rad_{X}(A\cup B)=rad_{X}(A),$ and
	\item[(ii)]\label{3.33(ii)} if $rad_{X}(A)=rad_{X}(B)$ and $(Cent_{X}(A)\backslash\Tilde{A})\cup (Cent_{X}(B)\backslash\Tilde{B})\neq\emptyset,$ then $Cent_{X}(A\cup B)=(Cent_{X}(A)\backslash\Tilde{A})\cup (Cent_{X}(B)\backslash\Tilde{B})$ and  $rad_{X}(A\cup B)=rad_{X}(A)=rad_{X}(B).$
\end{itemize} 
\end{theorem}

\begin{proof}
As $A$ and $B$ are separated, we get
$ \partial_{X}(A \cup B)= \partial_{X}(A) \cup \partial_{X}(B)$. So, for $a \in A,$ we have $\Psi(a,\partial_{X}(A\cup B))\leq \Psi(a,\partial_{X}(A))\leq rad_{X}(A).$ Similarly, for $b \in B,$ we have $ \Psi(b,\partial_{X}(A\cup B))\leq \Psi(b,\partial_{X}(B))\leq rad_{X}(B).$ This implies that $rad_{X}(A\cup B)\leq \max\{rad_{X}(A),rad_{X}(B)\}.$\par 
($i$) Let $a\in Cent_X(A\cup B)$. Then  $\Psi(a,\partial_{X}(A\cup B))\geq \Psi(b,\partial_{X}(A\cup B)),\forall b\in A\cup B.$ If $a\in A$, then $\Psi(a,\partial_{X}(A))\geq \min\{\Psi(b,\partial_{X}(A)),\Psi(b,\partial_{X}(B))\}, \forall b\in A$. So, either $a\in Cent_X(A)$ or $a\notin \Tilde{A}$. Given that  $Cent_{X}(A)\backslash\Tilde{A}\neq\emptyset$. If $a\in Cent_{X}(A)\backslash\Tilde{A}$, then $\Psi(a,\partial_{X}(A\cup B))=rad_X(A)$. Otherwise,  $\Psi(a,\partial_{X}(A\cup B))<rad_X(A)$.  Similarly, if $a\in B$, then either $a\in Cent_X(B)$ or $a\notin \Tilde{B}$, and in any case  $\Psi(a,\partial_{X}(A\cup B))<rad_X(B)<rad_X(A).$ Therefore, $Cent_X{(A\cup B)}= Cent_X(A)\backslash \Tilde{A}$ and $rad_{X}(A\cup B)=rad_{X}(A).$

$(ii)$ 
Similarly, for $a\in (Cent_{X}(A)\backslash \Tilde{A})\cup(Cent_{X}(B)\backslash \Tilde{B}),$ we get $\Psi(a,\partial_{X}(A\cup B))=rad_{X}(A).$ Otherwise, we get
$\Psi(a,\partial_{X}(A\cup B))<rad_{X}(A).$ Thus, $Cent_{X}(A\cup B)=(Cent_{X}(A)\backslash\Tilde{A}) \cup (Cent_{X}(B)\backslash\Tilde{B})$ and $rad_{X}(A\cup B)=rad_{X}(A).$
\end{proof}

In the above theorem, if $Cent_{X}(A)\backslash\Tilde{A}=\emptyset$, then it is easy to prove that the  radius of $A\cup B$ is not equals to the maximum radius of $A$ or $B$.  Interestingly, the way unions work in finite spaces is similar to how they behave in metric spaces when the subsets A and B are non-clopen and separated (Ref. \cite{Badra}).

\begin{theorem}\label{2.26}
Let A and B be  nonclopen separated subsets of a finite space $X.$
\begin{itemize}
	\item[(i)]\label{3.31(ii)} If $rad_{X}(A)>rad_{X}(B)$ and $Cent_{X}(A)\backslash\Tilde{A}=\emptyset,$ then $rad_{X}(A\cup B)<rad_{X}(A),$ and
	\item[(ii)]\label{3.31(iv)}  if $rad_{X}(A)=rad_{X}(B)$ and $(Cent_{X}(A)\backslash\Tilde{A})\cup (Cent_{X}(B)\backslash\Tilde{B})=\emptyset,$ then $rad_{X}(A\cup B)< rad_{X}(A)=rad_{X}(B).$ 
\end{itemize} 

\end{theorem}

By induction, we can generalise Theorem \ref{2.25}, for a finite union of subsets of a finite space.

\begin{theorem}\label{3.37}
	Let $X$ be a finite space. For nonclopen subsets $A_i \subseteq X, 1\leq i\leq n$, such that $A_i$ $\&$ $A_j$ are separated,
	for all $i\neq j$ and $n\in \mathbb{N},$
	let $\Tilde{A}_{j}=\{a\in Cent_{X}(A_{j})|
	\Psi(a,\partial_{X}(A_{i}))<rad_{X}(A_{j}),\text{ for some } i\neq j\}, 1\leq j\leq n$. Let $M$ be the collection of all those $A_j$ such that $rad_{X}(A_{j})= \max\{rad_{X}(A_{i})|1\leq i \leq n\}$ and $Cent_{X}(A_j)\backslash\Tilde{A_j}\neq\emptyset$. Then, $rad_{X}(\bigcup\limits_{1}^{n}A_{i}) \leq  \max\{rad_{X}(A_{i})|1\leq i\leq n\}$.\\
	Moreover, if $\bigcup\limits_{A_j\in M}(Cent_{X}(A_j)\backslash\Tilde{A_j})\neq\emptyset,$ 
	then  $Cent_{X}(\bigcup\limits_{1}^{n}A_{i}) = \bigcup\limits_{A_{j}\in M}(Cent_{X}(A_{j})\backslash \Tilde{A_{j}})$ $\&$ $rad_{X}(\bigcup\limits_{1}^{n}A_{i}) = \max\{rad_{X}(A_{i})|1\leq i \leq n\}.$
\end{theorem}


\subsection{\textbf{The Quasi-Center and the Quasi-Radius}}

Just like in metric spaces, for every subset of a finite space, we define quasi-center and quasi-radius.

\begin{definition}[Quasi-center of a subset]
	The quasi-center of $A$ is the set $\{a\in A|\Psi(a,A^c)\geq \Psi(b,A^c),\forall b\in A\},$ where $A^c$ denotes the complement of $A$ in $X$. We denote the quasi-center of $A$ in $X$ by $QCent_{X}(A).$\par
	Thus the quasi-center of $A$ is the set of all those elements of $A$ which are at the maximum furtherness from $A^c.$
\end{definition}

\begin{definition}[Quasi-radius of a subset]
	The quasi-radius of a subset $A$ of metric space $X$ is the furtherness of its quasi-center from its complement in $X.$ We denote the quasi-radius of $A$ in $X$ by $Qrad_{X}(A)$.\par
	Notice that $Qrad_{X}(A)=\Psi(QCent_{X}(A),A^c)=\Psi(a,A^c),\forall a\in QCent_{X}(A).$
\end{definition}

\begin{example}{\normalfont
		Let $X=\{a,b,c,d\}$ be a topological space with topology $\mathcal{T}=\{\emptyset,X,\{a\},\{c\},\{a,b\},\{c,d\},\{a,b,c\},\{a,c,d\}\}$. For $A=\{a,b\}\subseteq X$ is clopen. So, we have $\partial_{X}(A)=\emptyset$ and $Cent_X(A)=A$, $rad_{X}(A)=\infty$. But $A^c=\{c,d\}$ and $\Psi(a,A^c)=2, \Psi(b,A^c)=1$. So, $QCent_{X}(A)=\{a\}$ and $Qrad_{X}(A)=2$.
	}
\end{example}





We know that the largest open balls contained in a subset of a metric space are those balls which are centered at the quasi-center of that subset with  radius equals to the quasi-radius of that subset, see \cite{Badra}. Similarly, we can identify the largest forward open balls contained in any subset of a finite space.

\begin{theorem}
Let $A$ be a nonempty proper subset of a finite space $X$. Then the largest forward open balls of $X$ which are entirely contained in $A$ are the balls whose centers belong to $QCent_{X}(A)$ and radius is equal to $Qrad_{X}(A)$.
\end{theorem}

\end{document}